\documentclass[12pt,twoside,reqno]{amsart}

\allowdisplaybreaks
\newtheorem{thm}{Theorem}[section]
\newtheorem{lemma}{Lemma}[section]

\theoremstyle{definition}
\newtheorem{defn}[thm]{Definition}
\theoremstyle{remark}
\newtheorem{obs}[thm]{Remark}

\newcommand{\R}{{\mathbb R}}
\newcommand{\N}{{\mathbb N}}

\newcommand{\noi}{\noindent}

\numberwithin{equation}{section}

\begin{document}
\title[Decay of Solutions ]
{
 REGULARITY AND DECAY OF GLOBAL SOLUTIONS FOR THE 4D NAVIER-STOKES EQUATIONS POSED ON SMOOTH DOMAINS.
}
\author{N. A. Larkin\;	\&
	M. V. Padilha$^\dag$}
%\author[N.~A. Larkin]{N.~ A. Larkin$^{\dag}\;\; \&$} \author[M. ~V. Padilha]{M. ~V. Padilha}
%\author[M. ~V. Padilha]{M. ~V. Padilha}

\address
{
	Departamento de Matem\'atica, Universidade Estadual
	de Maring\'a, Av. Colombo 5790: Ag\^encia UEM, 87020-900, Maring\'a, PR, Brazil
}

\thanks
{
	$^\dag$
	Corresponding author 
}
\bigskip

\email{ nlarkine@uem.br;nlarkine@yahoo.com.br; marcos.padilha@ifpr.edu.br}
\date{}

\begin{abstract}We consider an initial-boundary value problem for the 4D  Navier-Stokes equations posed on bounded smooth domains. We prove the existence and uniqueness of regular solutions as well as their exponential decay and additional regularity properties have been established assuming restrictions on initial data.
\end{abstract}

\maketitle

\section{Introduction}\label{introduction}

In this work, we study an initial-boundary value problem for the 4D Navier-Stokes equations, 
\begin{align}\label{1-1}
	& u_t - \nu \Delta u  +(u\cdot \nabla)u + \nabla p = 0,\;\; \mbox{in} \;\; \Omega\times \R,\\
	&\nabla \cdot u=0 \;\;\mbox{in}\;\; \Omega,\;\;\;u|_{\partial \Omega}=0,\label{1-2}\\
	&u(x,0)=u_0(x),\label{1-3}
\end{align}

\noi where $x = (x_1,x_2,x_3,x_4) \in \Omega \subset \R^4$ and $\Omega$ is a bounded smooth domain.\\
Creation of the mathematical theory of the Navier-Stokes equations started with the famous Leray`s paper, see \cite{Leray}, where he studied  laminar (slow) and turbulent flows of the viscous liquid, defined weak and regular solutions to the three-dimensional Cauchy problem and proved the existence of a weak solution. This work oriented mathematicians to study existence, uniqueness and decay of the energy  of solutions to the multi-dimensional Navier-Stokes equations \cite{Salsa}. The most advanced results were established in the two-dimensional case, where the existence of weak and regular global solutions,  uniqueness and decay of the energy for bounded and unbounded Lipschitz domains were established without additional conditions of smallness for the initial data, see \cite{Lady},  Theorem 2, \cite{Hermann} Chapter V, Theorem 4.2.1,  \cite{Lions} Chapter I, Theorems 6.2, 6.3, 6.6; \cite{Temam} Chapter III, $\S 3.3$, Theorem 3.1; \cite{Larkin3} Theorems 4.2, 4.5, 4.6. The situation is more complicated for three dimensions: to prove uniqueness and regularity of global solutions, some authors considered small initial data, see
\cite{Lady}, Theorem 1,  Theorem 3 ; \cite{Lions}, Chapter I, Theorem 6.6; \cite{Temam}, Chapter III, Theorems 3.2, 3.5, 3.6; while the others studied initial boundary value problems in  ``thin" domains, see  \cite{Raugel}. In \cite{Larkin3D}, regularity, uniqueness and exponential decay of the energy has been proved making use of connection between the geometric sizes of a domain and initial data  without involving the first eigenvalue of the projector operator of the Laplacian that was for a long time the standard approach (see  \cite{Fo3}).

The existence of weak and strong global solutions to (\ref{1-1})-(\ref{1-3}) for a dimension $n \leq  4$ have been considered in \cite{Lions}, Chapter I, Theorem 6.8; \cite{Temam}, Chapter III, $\S 3.1$ under assumptions of small initial data and the continuity of solution was studied unders conditions \cite{Scheffer}.  Recent works approach the criticality of the dimension 4 and establish criteria over spaces or domains. In \cite{Li} it is shown that for dimensions bigger than three, blow up for some finite time. In \cite{Larkin4D}, the existence, uniqueness and exponential decay of global strong solution to (1.1)-(1.3) posed on Lipschitz bounded and unbounded domains have been established assuming some connections between the geometric sizes of domains and initial data. \\ 
Our goal here is to consider the 4D Navier-Stokes system posed on smooth domains. Earlier, the problem (1.1)-(1.3) has been considered posed on parallelepipeds, bounded and unbounded \cite{Larkin4D} and it was revealed  that the fourth dimension is a critical one in the sense that the same method can not be applied  to prove the existence of a solution in the $H^2$-norm. Alternatively, our approach, based on a special basis, allows to prove the existence and uniqueness of regular solutions as well as their exponential decay.

\section{Notations and Auxiliary Facts}

Let $\Omega$ be a domain in $\R^4$ and $x = (x_1, x_2, x_3, x_4) \in \Omega$. We use the standard notations of Sobolev spaces $W^{k,p}$, $L^p$ and $H^k$ for vector functions and the following notations for the norms \cite{Evans, Lions, Temam}

(i) For vector functions $u(x,t) = (u_1(x,t), u_2(x,t), u_3(x,t), u_4(x,t)),$

$$\| u \|^2 = \sum_{i=1}^4\int_{\Omega}  | u_i |^2  dx,\;
\| u \|_{L^p(\Omega)}^p = \sum_{i=1}^4\int_{\Omega}  | u_i |^p  dx ,$$$$
\| u \|_{W^{k,p}(\Omega)}^p = \sum_{0 \leq \alpha \leq k} \sum_{i=1}^4 \|D^\alpha u_1 \|_{L^p(\Omega)}^p ,$$

When $p = 2$, $W^{k,p}(\Omega) = H^k(\Omega)$ is a Hilbert space with the scalar product

$$((u,v))_{H^k(\Omega)}=\sum_{|j|\leq k}(D^ju,D^jv),\;
\|u\|_{L^{\infty}(\Omega)}=\text{ess}\; \text{sup}_{\Omega}|u(x)|.$$

We use the $H_0^k(\Omega)$ to represent the closure of $C_0^\infty(\Omega)$, it is, the set of all $C^\infty$ functions with compact support in $\Omega$, with the norm of $H^k(\Omega)$ \cite{Lions}.

(ii) For scalar functions $f(x,t)$

$$\| f \|^2 = \int_{\Omega} | f |^2d\Omega, \hspace{1cm} \| f \|_{L^p(\Omega)}^p = \int_{\Omega} | f  |^p\, d\Omega,$$
$$\| f \|_{W^{k,p}(\Omega)}^p = \sum_{0 \leq \alpha \leq k} \|D^\alpha f \|_{L^p(\Omega)}^p, \hspace{1cm} \| f \|_{H^k(\Omega)} = \| f \|_{W^{k,2}(\Omega)}^2.$$

Define the auxiliary spaces: (see \cite{Temam})

$$D(\Omega) = \{ f \in C^\infty(\Omega); \text{supp f is a compact set of }\Omega\}$$

$$\mathcal{V} = \{ u \in \mathcal{D}(\Omega), \nabla \cdot u = 0 \},$$

$$V = \text{ the closure of } \mathcal{V} \text { in } H_0^1(\Omega),$$

\noi Obviously, $V$ is a subspace of $H_0^1(\Omega)$. 

$$H = \text{ the closure of } \mathcal{V} \text { in } L^2(\Omega).$$

\noi The space $H$ is eqquiped with the natural $L^2$ inner product. The space $V$ will be equipped with the scalar produt 

$$((u,v)) = \sum_{i =1}^4 (D_i u, D_i v)$$

\noi when $\Omega$ is bounded. If $\Omega$ is unbounded, we define the inner product as the sum of the inner products as following:

$$[[u,v]] = (u,v) + ((u,v)).$$

The next Lemmas will be used in estimates:

\begin{lemma}\label{steklov}[Steklov's Inequality \cite{steklov}] Let $v \in H^1_0(0,L).$ Then
	\begin{equation}\label{Estek}
	\frac {\pi^2}{L^2}\|v\|^2(t) \leq \|v_x\|^2(t).
	\end{equation}
\end{lemma}

\begin{proof}
Let $v(t)\in H^1_0(0,\pi)$, then by the Fourier series ,  $$\int_0^{\pi}v_t^2(t)\,dt\geq\int_0^{\pi}v^2(t)\,dt.$$
Inequality \eqref{Estek} follows by a simple scaling.
\end{proof}

\begin{lemma}[See \cite{Lady}; Theorems 7.1,-7.3; \cite{Magenes}, Theorem 2.2.] \label{lemma3.3}  Let $v \in H_0^1(\Omega)$ and $n = 4$, then
	
	\begin{equation}
	\| v \|_{L^4(\Omega)} \leq 3 \| \nabla v \|_{L^2(\Omega)}.
	\end{equation}
 If $v\in H^1(\Omega),$ then
 \begin{equation}
 	\|v\|_{L^4(\Omega)}\leq C_1(\Omega)\|\nabla v\|_{L^2(\Omega)}.
 \end{equation}
\end{lemma}

\begin{obs}The constant $C_1(\Omega)$ in Lemma 2.2 depends on $\Omega$.
\end{obs}

\begin{lemma}[See: \cite{Temam}]\label{lemmabuvw}Let $b(u,v,w) = ((u \cdot \nabla)v,w),$ and $n= 4$, then 
	$$|b(u,v,w)| \leq 9 \|u\|_V \|v\|_V  \|w\|_V.$$
	
\end{lemma}

\begin{proof}By the H\"{o}lder Inequality, we calculate
	
	\begin{align*}|b(u,v,w)| \leq \sum_{i,j =1}^4 \int_\Omega u_i D_i v_j w_j\, dx 
		\leq \sum_{i,j=1}^4 \| u_i \|_{L^4(\Omega)} \| D_i v_j \| \| w_j \|_{L^4(\Omega)}\\
		\leq \Big(\sum_{i=1}^4 \| u_{i} \|_{L^4(\Omega)}\Big)^{1/4} \Big( \sum_{i,j = 1}^4 \|D_i v_j \|^2 \Big)^{1/2} \Big( \sum_{j = 1}^4 \| w_j \|_{L^4(\Omega)}^2\Big)^{1/4}\\
		\leq \|u \|_{L^4(\Omega)} \|v\|_V \| w \|_{L^4(\Omega)}
		\leq 9 \|u\|_V \|v\|_V  \|w\|_V.
	\end{align*}

The proof of Lemma 2.3 is complete.

\end{proof}

\begin{lemma}[See: \cite{Cattabriga}]\label{Cattabriga}
	
	Let $\Omega$ be a bounded set of  class $C^2$. Let us suppose that
	
	\begin{equation}
		\label{catta1}u \in W^{1,\alpha}(\Omega), \hspace{5mm} p \in L^\alpha(\Omega), \hspace{5mm} 1 < \alpha < +\infty
	\end{equation}
	
	\noi are solutions of the generalized Stokes problem:
	\begin{equation}\begin{cases}
			-\nu\Delta u + \nabla p = f, \text{ in } \Omega,\\
			\nabla \cdot u = 0, \text { in } \Omega,\\
			u = 0 \text{ on } \partial \Omega.
	\end{cases}\end{equation}
	
	If $f \in L^{\alpha}(\Omega)$,  then $u \in W^{2,\alpha}(\Omega)$, $p \in W^{1, \alpha}(\Omega)$ and there is a constant $C(\alpha, \nu,\Omega)$ such that
	\begin{align}&\| u \|_{W^{2,\alpha}(\Omega)} + \| p \|_{W^{1, \alpha}(\Omega)} \leq C \big\{ \| f\|_{L^{\alpha}(\Omega)}      \big\}.
	\end{align}

\end{lemma}

\begin{lemma}[See: \cite{Temam}, Proposition 2.2.]
	
	Let $\Omega$ be an open bounded set of  class $C^r; r=max(m,2), m >0\;\text{is integer} $. Suppose that
	
	\begin{equation}
		\label{catta1}u \in W^{2,\alpha}(\Omega), \hspace{5mm} p \in W^{1,\alpha}(\Omega), \hspace{5mm} 1 < \alpha < +\infty
	\end{equation}
	
\noi are solutions to the generalized Stokes problem:
		\begin{equation}\begin{cases}
			-\nu\Delta u + \nabla p = f, \text{ in } \Omega,\\
			\nabla \cdot u = 0, \text { in } \Omega,\\
			u = 0 \text{ on } \partial \Omega.
	\end{cases}\end{equation}
	
	If $f \in W^{m,\alpha}(\Omega)$,  then $u \in W^{m+2,\alpha}(\Omega)$, $p \in W^{m+1, \alpha}(\Omega)$ and there is a constant $C(\alpha, \nu,\Omega)$ such that
		\begin{align}&\| u \|_{W^{m+2,\alpha}(\Omega)} + \| p \|_{W^{m+1, \alpha}(\Omega)} \leq C \big\{ \| f\|_{W^{m,\alpha}(\Omega)}      \big\}.
	\end{align}

\end{lemma}

\begin{defn}An open subset $\Omega$ is said to satisfy the {\bf Cone Condition} if there exists a finite cone $C$ such that each $x \in \Omega$ is the vertex of a finite cone $C_x \subset \Omega$ and congruent to $C$. 
\end{defn}

\begin{lemma}[Sobolev Embeddings. See \cite{Brezis}] Let $\Omega \subset \R^4$ be a domain which satisfies the cone condition, we have the following embeddings.

\begin{enumerate}
	\item[(i)] If $mp < 4$, then

$$W^{m,p}(\Omega) \hookrightarrow L^q(\Omega),$$

\noi for $p\leq q \leq kp / (4-mp)$.

\item[(ii)]If $mp = 4$, then 

$$W^{m,p}(\Omega) \hookrightarrow L^q(\Omega),$$

\noi for $p \leq q < \infty$.

\end{enumerate}

	\end{lemma}

\section{Existence and Uniqueness}

Let $\Omega \subset \R^4$ be a bounded smooth domain. Without loss of generality, we assume that $\Omega \subset (\R^+)^4$ and its border is tangent to the axis $x_i$. Consider the following system:
\begin{align}
	& u_t+(u\cdot \nabla)u - \nu \Delta u + \nabla p = 0,\;\; \mbox{in} \;\; \Omega\times (0,t),\label{NS1}\\
	&\nabla \cdot u=0 \;\;\mbox{in}\;\; \Omega,\;\;\;u|_{\partial \Omega}=0\label{NS2},\\
	&u(x,0)=u_0(x)\label{NS3},
\end{align}

\noi where $u_0(x) \in  V$. Let $L_i$ be positive numbers for $i= 1,\dots,4,$ where $D = \{ x = (x_1,x_2,x_3,x_4); 0 < x_i < L_i\}$ is the minimal parallelepiped which contains $\Omega$.  

\begin{defn}\label{defsol}A vector function $u(x,t)$ is a regular solution of (\ref{NS1})-(\ref{NS3}) such that
	
	\begin{equation}\label{SS1}u \in L^\infty(\R^+;V)\cap L^2(\R^+;H^2(\Omega)),\end{equation}
	\begin{equation}\label{SS2}u_t \in L^2(\R^+; L^2(\Omega)),\end{equation}
	$$\nabla \cdot u = 0, \, \, \, \, u|_{\partial \Omega} = 0, \,\,\,\,  \, \, \, \, u(x,0) = u_0(x)$$	
	
\noi satisfying the following identity:

\begin{equation}\label{generalized}\int_{\R^+}\int_{\Omega} \Big( u_t  - \nu \Delta u + (u\cdot \nabla)u \Big)\phi\, dxdt =0\end{equation}

\noi for every $\phi \in V$.

\end{defn}

A solution in Definition \ref{defsol} is equivalent to a solution of the variational problem (see \cite{Temam}), i.e.

\begin{align}\label{V-1}&u_t + Au + Bu = 0 \text{ in } (0,t), t > 0;\\
		&u(0) = u_0,
	\end{align}

\noi where $A=-P\Delta$ is an operator from $V$ to $V^{\prime}$ and $P$ is the orthogonal projection on $H$ in $L^2(\Omega)^4$\; such that $$(Au,v) = \nu((u,v))\;\;\text{and}\;B(v,w)=P[(v \cdot\nabla)w]$$ 
defined for $u,v,w\in V$ and $H$- valued. Moreover,
$$b(u,v,w)=(B(u,v),w),\;\;B(u)=B(u,u).$$
Operators $A$ and $B$ may be extended by continuity to linear (bilinear) operators from $V$ (respectively \;$V\times V$) into the dual $V^{\prime}\supset H\supset V$ of $V,$ ( see \cite{Fo3}.)
\begin{equation}\label{defBU}\langle B(u),v\rangle = b(u,u,v),\end{equation}

\noi where $b(u,v,w)$ is a trilinear form defined as, see \cite{Temam},

$$b(u,v,w) = \sum_{i,j=1}^4 \int_\Omega u_i(D_i v_j)w_j \,dx.$$

The stationary problem

\begin{equation}
	\nu((u,v)) = (f,v), \, \, \forall v \in V, 
\end{equation}

\noi with $f \in L^2(\Omega)$ has a unique solution (See \cite{Temam} P. 23). One can check that the mapping $\Lambda$ such that

$$(\Lambda f, v) = \nu ((u,v)),$$

\noi is a self-adjoint  linear operator. Hence  there exists a sequence of eigenfunctions $w_i$ and eigenvalues $\lambda_i$ such thats $\Lambda w_j = \lambda_j w_j$, 
$$((w_j,v)) = \lambda_j (w_j, v), \forall v \in V$$

\noi and $w_j \in V$, $j \geq 1$, $\lambda_j > 0$ and $\lambda_j \to \infty$ as $j \to \infty$.  See  \cite{Lions}, Chapter I, $\S$ 6.3; \cite{Temam}, Chapter III, Lemma 3.7.

\begin{thm}\label{T3-1}Let $u_0 \in  V$ such that
	
	\begin{equation}\label{hyp}  \nu -max(9,3C_1(\Omega)) \| u_0\|_V > 0,\end{equation}

then there exists a unique regular solution $u(x,t)$   to (\ref{NS1})-(\ref{NS3}) in the sense of Definition \ref{defsol} . Moreover,

\begin{align}&\| u \|^2(t) \leq \| u_0 \|^2 e^{-2\nu\chi t},\label{T1-R1}\\
	&\| u \|_V^2(t) \leq \| u_0 \|_V^2 e^{-\nu\chi t}. \label{T1-R2}
\end{align}

\noi where $\chi = \sum_{i = 1}^4 \Big(\frac{\pi^2}{L_i^2}\Big)$.

	\end{thm}

\begin{proof}Here we follow the scheme proposed in \cite{Lions, Temam}.  Let $\{ w_i \}$ be the eigenfunctions defined above. These functions span the set $V$ and $((w_j, w_k)) = \lambda_j \delta_{jk}$ for all $j, k \in \N$. 
	We have
	\begin{equation}\label{SB3}\lambda_j(v,w_j) = (v, Aw_j) = ((v,w_j)).\end{equation}
	
The next steps follow  the Faedo-Galerkin's method. 

For $m \in \N,$ let $V_m$ be the finite dimensional linear span of $\{w_i\}$. Define the function

\begin{equation}\label{umbygm}u_m(x,t) = \sum_{i=1}^m g_{im} (t) w_i,\end{equation}

\noi where $g_{im}$ is the unique solution of the approximate problem:

\begin{equation}\label{app1}\begin{cases}(u_m'(t), w) + \nu ((u_m(t), w)) + b(u_m(t), u_m(t), w) = 0,
		\\	u_m(0) = u_{0m}
		\text{ for } t \in [0,T].\\
\end{cases}\end{equation}

\noi Here $w \in V_m$ and $u_{0m}$ is the Gramm-Schmidt projection of $u_0$ onto the space $ V_m$.

	{\bf Estimate I:} Multiply (3.16) by $u_m$ and integrate over $\Omega$	
	\begin{equation}\label{3-0001}(u_m', u_m)(t) + (Au_m, u_m)(t) = 0\end{equation}
\noi to obtain
	\begin{align}\label{3-001}\frac{d}{dt} \|u_m \|^2(t) + 2\nu \| u_m \|_V^2(t) = 0.\end{align}
Define $\tilde{u_m}(x,t)$ an extension of $u_m(x,t)$ to the domain $D$ as  follows
\begin{equation}\label{3-000001}\tilde{u_m}(x,t) = \begin{cases}
		u_m(x,t) \text{ if } x \in \Omega,\\
		0 \text{ if } x \in D / \Omega.\end{cases}\end{equation}
	By Lemma \ref{steklov}, 
	\begin{equation}\label{3-01}\sum_{i=1}^4 \Big(\frac{\pi^2}{L_i^2}\Big)\| u_m \|^2 = \sum_{i=1}^4 \Big(\frac{\pi^2}{L_i^2}\Big) \| \tilde{u_m} \|^2 \leq \| \tilde{u_m} \|_V^2 = \| u_m \|_V^2.\end{equation}
	Therefore, \eqref{3-01} and \eqref{3-001} provide
\begin{align}\label{3-1}	\frac{d}{dt} \| u_m \|^2(t) + 2\nu\chi \|u_m\|^2(t) \leq 0,\end{align}
\noi where $\chi = \sum_{i=1}^4 \frac{\pi^2}{L_i^2}$.  Integrating \eqref{3-1}, we obtain
\begin{align}\label{3-2}\|u_m\|^2(t) \leq \|u_0\|^2 e^{-2\nu\chi t}.\end{align}
	Integrating \eqref{3-001} in $\R^+$, we get
\begin{align}\label{3-02}2\nu \int_0^\infty \| u_m \|_V^2(t) \, dt \leq \| u_0\|^2.\end{align}

Clearly, the estimates \eqref{3-2} and \eqref{3-02} hold because $\|u_{0m} \| \leq \|u_0\|$. 
In summary, the estimates \eqref{3-2} and \eqref{3-02} mean that 

\begin{equation}\label{bound1}
	u_m \text{ is bounded in } L^\infty(\R^+, H)\cap L^2(\R^+, V). 
\end{equation}

{\bf Estimate II:} Choose $w = w_j$ in (3.16) and multiply by $\lambda_j$ to get
\begin{equation}\label{SB2} ((u_m', w_j)) + \nu (Au_m, Aw_j) + (Bu_m, Aw_j) = 0.\end{equation}. 

Sum up for $j=1,\dots,m$ to obtain

\begin{equation}\label{SB0}
	((u_m',u_m))+ \nu (Au_m,Au_m) + (Bu_m, Au_m) = 0
\end{equation}
\noi and consequently,
\begin{align}
	\frac{1}{2} \dfrac{d}{dt} \|u_m\|_V^2+\nu\|Au_m\|^2+b(u_m,u_m,Au_m)=0.
\end{align}

We estimate 

\begin{align}| b(u_m,u_m,Au_m) | \leq \| Au_m \| \| u_m \|_{L^4(\Omega)} \| \nabla u_m \|_{L^4(\Omega)}.
\end{align}

Using (2.3) from  Lemma \eqref{lemma3.3}, we get  $$\| \nabla u_m \|_{L^4(\Omega)} \leq C_1(\Omega) \| \nabla u_m \|_{V} \leq  C_1(\Omega)\| Au_m \|,$$ and finally,
\begin{align}| b(u_m,u_m,Au_m) | \leq 3C_1 \| u_m \|_V(t) \| Au_m \|^2(t). 
	\end{align}

Hence

\begin{align}\frac{d}{dt} \| u_m \|_V^2(t) + \nu \| Au_m\|^2(t) + \Big(\nu - 3C_1\| u \|_V(t)\Big) \| Au\|^2(t) \leq 0.
	\end{align}

By \eqref{hyp}, it follows, see \cite {Temam}, Ch. III, p. 305,
\begin{align}\label{n01}\frac{d}{dt} \| u_m \|_V^2(t) + \nu \| Au_m \|^2(t) \leq 0, \forall t > 0.
	\end{align} 

Integrating \eqref{n01}, we get

\begin{align}\label{3-6}\| u_m\|_V^2(t) + \nu \int_0^t \| A u_m \|^2(s)\, ds \leq \| u_0 \|_V^2.
	\end{align}

On the other hand, 
\begin{align}\chi\|u_m\|^2(t)\leq\| u_m \|_V^2(t)  = (Au_m,u_m)(t) \leq \|Au_m\|(t) \|u_m\|(t).
\end{align}

Hence,

\begin{align}\chi\|u_m\|(t) \leq \|Au_m\|(t) 
\end{align}
and 

\begin{align} \| u_m \|_V^2(t) \leq \frac{1}{\chi}\|Au_m\|^2(t). 
\end{align}

This reduces (3.29) to the form

\begin{align}\frac{d}{dt} \| u_m\|_V^2(t) +\nu\chi \| u_m \|_V^2(t) \leq 0,
	\end{align}

Hence
\begin{equation}\label{es1}
	\| u \|_V^2(t) \leq \| u_0 \|_V^2 e^{-\nu\chi t}.
\end{equation}
Note that \eqref{3-6} implies that $Au \in L^2(\R^+; L^2(\Omega))$. Then, for a non-singular $t_0 \in \R^+$ we calculate 

\begin{align} \| \Delta u \|(t_0)  \leq C\|Au\|(t_0)
\end{align}

and  conclude that $u \in L^2(\R^+; H^2(\Omega))$. By Lemma \ref{lemma3.3}, we  estimate

\begin{align}\| Bu \|(t_0) \leq \| u \|_{H^2(\Omega)}(t_0).
	\end{align}

\noi It follows from  \eqref{V-1}
\begin{equation}\label{bound2}
	\frac{\partial}{\partial t}u_m (t)\text{ is bounded in }L^2(\R^+;L^2(\Omega)).
	\end{equation}

{\bf Passage to the Limit.} Since the estimates for \eqref{bound1} and \eqref{bound2} do not depend on $m$ or $t$, there exists a subsequence $u_m(x,t)$ which converges weakly to a function $u(x,t)$ in their spaces. To prove that such  $u(x,y)$ solves (\ref{NS1})-(\ref{NS3}), we use the definition of weakly convergence :\\

since $\dfrac{\partial}{\partial t} u_{m}$ is bounded in $L^2 (\R^+; L^2(\Omega))$, then
\begin{equation}\label{conv1}
\int_{\R^{+}}	\int_D u_{mt} \phi \, dx \, dt \to \int_{\R^{+}} \int_D u_t \phi \, dx\, dt.
\end{equation}

Since $u_m$ is bounded in $L^\infty(\R^+; V),$ then for $a.e.\;t>0$
\begin{equation}\label{conv2}
	\int_D \nu u_{mx_i} \phi_{x_i} \, dx \to \int_D \nu u_{x_i} \phi_{x_i}\, dx
\end{equation}
\noi and
\begin{equation}\label{conv3}
	\int_D u_{mi} u \phi_{x_i} \, dx \to \int_D u_i u \phi_{x_i} \, dx.
\end{equation}

By \eqref{conv1}, \eqref{conv2} and \eqref{conv3}, we conclude that $u(x,t)$ satifies \eqref{defsol}.
This proves the existence part of Theorem 3.2.

\begin{lemma}\label{uniq}[Uniqueness] There is at maximum one solution $u(x,t)$ of \eqref{NS1}-\eqref{NS3} such that
	$$u \in L^\infty(\R^+; V).$$
\label{uniquenesslemma}\end{lemma}
	
\begin{proof}	
	 Suppose that $u_1$ and $u_2$ are two different solutions such as obtained in the class \eqref{SS1}. We write $u = u_1 - u_2$, which solves 
\begin{align}
	& u_t+ Au + B(u_1) - B(u_2) = 0,\;\; \mbox{in} \;\; \Omega\times (0,t),\label{NS01}\\
	&u(x,0)=0 \label{NS02},
\end{align}

Multipying (\ref{NS01})-(\ref{NS02}) by $u$, we obtain

\begin{equation}\label{uni1}
	\frac{d}{dt} \| u \|^2 (t) + 2\nu \| u \|_V^2 (t) = 2b(u_2,u_2, u) - 2b(u_1, u_1,u),
\end{equation}
Since
$$b(u,v,w) = \sum_{i,j=1}^4 \int_\Omega u_i(D_i v_j)w_j \,dx,$$
then $b(u_i,u_j,u_j,)=0, \; i,j=1, 2.$ One can check that

\begin{align*}2b(u_2,u_2, u) - 2b(u_1, u_1,u) = 2b(u_2, u, u_2)  -2 b(u_1, u, u_1)\\
	= 2b(u_2,u, u_2) -2b(u_1, u,u_1+u_2-u_2) =-2 b(u,u, u_2).
	\end{align*}

Making use of the H\"{o}lder inequality and Lemma 2.3, (3.46) becomes

\begin{align}
	\frac{d}{dt} \| u \|^2 (t) + 2\nu \| u \|_V^2 (t) = -2b(u,u,u_2)\notag \\
	\leq \| u \|_{L^4(\Omega)} \| u \|_V \| u_2 \|_{L^4(\Omega)} \leq 9\| u \|_V^2  \|\nabla  u _2\| .
\end{align}
This implies
$$	\frac{d}{dt} \| u \|^2 (t) + 2[\nu -9\| u _2\|_V(t)]\|u\|_{V}^2(t)\leq 0.$$ 
Condition (3.11) guarantees,   see \cite {Temam}, Ch. III, p. 305, that

\begin{equation}\frac{d}{dt} \| u \|^2(t) \leq 0 . 
\end{equation}

\noi This implies

\begin{equation}
	\| u \|(t)^2 \equiv 0
\end{equation}

\noi and $u_1 = u_2$ what proves Lemma \ref{uniquenesslemma}.
\end{proof}

The existence part and Lemma \ref{uniquenesslemma} prove Theorem 3.2. 
\end{proof}

{\bf More Regularity} 

\begin{thm}
	Under the conditions of Theorem \ref{T3-1}, if additionally $u_0 \in H^2(\Omega)$ and 
	
	\begin{equation}
		\label{hyp2}\nu - 9\| u_0 \|_V \geq 0,
	\end{equation}
	
\noi then
	
	 \begin{align}\label{T2-R}
	 u &\in L^\infty (\R^+; H^2(\Omega))\cap L^2(\R^+; H^3(\Omega)),\\
	 	u_t &\in L^\infty (\R^+; L^2(\Omega))\cap L^2(\R^+; H^1(\Omega)).
	 \end{align}
\end{thm}
\begin{proof} 
Derivate \eqref{NS1} with respect  to $t$, multiply by $u_t$ and integrate over $\Omega$ to obtain 

\begin{equation}
	\frac{d}{dt} \| u_t \|^2(t) + \nu \| u_t \|_V^2(t) + b(u_t,u,u_t) = 0.
\end{equation}

Making use of the H\"{o}lder inequalities and Lemma 2.2, it can be rewritten as

\begin{align}
\dfrac{d}{dt}	\| u_t \|^2(t) + \nu \| u_t \|_V^2(t) \leq | b(u_t,u,u_t) | (t)	\leq 9 \| u_t \|_V^2(t) \| u \|_V(t).
	\end{align}
 
This implies
\begin{equation}
\frac{d}{dt} \| u_t\|^2(t) +  \Big(2\nu - 9 \| u\|_V(t)\Big) \| u_t \|_V(t)^2 \leq 0.
\end{equation}

By (3.11) and \eqref{hyp2},
\begin{align}\frac{d}{dt} \| u_t \|^2(t)  + \nu \| u_t \|_V^2 \leq 0,
	\end{align}

Using Steklov's inequality, integrating over $(0,t)$, we obtain

\begin{equation}\label{T2-05}
	\| u_t \|^2(t) + \int_0^t \nu\|u_t\|_V^2 \leq \| u_t\|^2(0).
\end{equation}

It remains to show that $\| u_t \|(0)$ as a limit, is bounded. Multiplying \eqref{NS1} by $u_t$ and integrating over $\Omega$, we get

\begin{equation}
	\| u_t\|^2 + (Au, u_t) + B(u,u,u_t) = 0, 
\end{equation}
then 
\begin{equation}\label{T2-03}
	\| u_t \|(0) \leq \| Au_0 \| + \| Bu_0 \|,
\end{equation}

and since $u_0 \in H^2(\Omega)$, $u_t(0)$ belongs to $L^2(\Omega)$ and by \eqref{T2-05}, we have $u_t \in L^2(\R^+; H^1(\Omega))$.

Let $u(x,y)$ be a solution obtained in Theorem \eqref{3-1}, write \eqref{NS1} as \eqref{V-1}, 
 
$$Au = -u_t - Bu.$$
Making use of (3.57) and Theorem 3.2, it can be seen that $u_t+Bu \in L^{\infty}(\R^+; L^2(\Omega)).$ Hence, by Lemma 2.4, 
\begin{equation}
	u\in L^{\infty}(\R^+;H^2(\Omega)).\end{equation}
Note that, by H\"{o}lder's inequality,
\begin{align} \label{T2-07} \| (u\cdot \nabla ) u \|_{H^1(\Omega)} \leq \| (u\cdot \nabla )u \| + \| (\nabla u)^2 \| + \| u D^2 u \|\notag \\
	\leq C \| u \| \| \nabla u \| + \| \nabla u \|_{L^4(\Omega)}^2 + \| u \|_{L^6(\Omega)}  \| D^2 u \|_{L^{3}(\Omega)} 
\end{align}

By Lemma \ref{lemma3.3}, the first two terms are estimated by $\| u \|_{H^2(\Omega)}$ such that is bounded in $L^2(\R^+)$, and it remains to show that for every non-singular $t_0$, $\| u \|_{L^6(\Omega)}(t_0) \|D^2 u \|_{L^3(\Omega)}(t_0)$ is bounded.

\noi We have
\begin{align}| \langle Bu,v\rangle |(t_0)  \leq \| Du \|_{L^4(\Omega)}(t_0)  \| u \|_{L^{12}(\Omega)}(t_0)  \| v \|_{L^{3/2}(\Omega)}(t_0).  \end{align}

\noi This implies

\begin{align}\label{T2-06} \| Bu \|_{L^3(\Omega)}(t_0)  \leq \|Du \|_{L^4(\Omega)}(t_0)  \| u \|_{L^{12}(\Omega)}(t_0)  
	\end{align}

By Sobolev embeddings, 
\begin{align}\| u \|_{L^{12}(\Omega)}(t_0) \leq \| u \|_{H^2(\Omega)}(t_0),
\end{align}

Jointly (3.63) and (3.64)  give
\begin{align}\| (u \cdot \nabla ) u \|(t_0) = \| Bu \|_{L^3(\Omega)}(t_0) \leq C \| u \|_{H^2(\Omega)}(t_0).
	\end{align}

Concluding the estimate in \eqref{T2-07} with \eqref{T2-06}, we obtain
\begin{equation}
	\| Bu \|_{H^1(\Omega)}(t_0)= \| (u\cdot \nabla) u \|_{H^1(\Omega)}(t_0) \leq C \| u \|_{H^2(\Omega)}^2(t_0).
\end{equation}

Hence, for every non-singular $t_0 \in \R+$, the right-hand side of \eqref{V-1} belongs to $L^2(\R^+; H^1(\Omega))$ and we obtain that $Au \in L^2(\R^+; H^1(\Omega))$. By Lemma 2.5, we  get $u \in L^2(\R^+; H^3(\Omega))$. This, (3.60) and (3.57) complete the proof of Theorem 3.3.
\end{proof}
\medskip

{ \bf Conclusions }

We have considered an initial-boundary value problem for the 4D Navier-Stokes equations with restrictions on the initial data. Existence, uniqueness and exponential decay of regular solutions were established. For more regular initial data (for example, $H^2(\Omega)),$ it was possible to prove more more regularity of global solutions. 

\medskip
\medskip

\bibliographystyle{torresmo}

\end{document}